\documentclass[11pt,reqno]{amsart}
\usepackage{amssymb}
\usepackage[all]{xy}
\usepackage{fancyhdr}
\usepackage{hyperref}
\usepackage{tikz-cd}
\usepackage{cite}
\usepackage{amsmath,amsfonts}
\usepackage{graphicx}
\usepackage{wrapfig}
\usepackage{array}
\usepackage{amsthm}
\usepackage[left=1.2in, right=1.2in, top=1in, bottom=1in]{geometry}
\usepackage{etoolbox}
\patchcmd{\section}{\scshape}{\bfseries}{}{}
\makeatletter
\renewcommand{\@secnumfont}{\bfseries}
\makeatother

\DeclareMathOperator{\Hom}{Hom}
\DeclareMathOperator{\Spec}{Spec}

\DeclareMathOperator{\Img}{Img}

\DeclareMathOperator{\Pic}{Pic}

\DeclareMathOperator{\Proj}{Proj}

\input xy
\xyoption{all}
\usepackage{secdot}  

\theoremstyle{plain}
\newtheorem{mydef}{\textbf{Definition}}[section]
\newtheorem{myeg}[mydef]{\textbf{Example}}
\newtheorem{mythm}[mydef]{\textbf{Theorem}}
\newtheorem*{nothm}{\textbf{Theorem}}

\newtheorem{rmk}[mydef]{\textbf{Remark}}
\newtheorem{lem}[mydef]{\textbf{Lemma}}
\newtheorem{pro}[mydef]{\textbf{Proposition}}

\newtheorem{cor}[mydef]{\textbf{Corollary}}
\patchcmd{\abstract}{\scshape\abstractname}{\normalsize{\textbf{\abstractname}}}{}{}
\begin{document}

\title{\v{C}ech cohomology of semiring schemes}


\author{Jaiung Jun}
\address{Department of Mathematics, Johns Hopkins University, Baltimore, MD 21218 USA}
\curraddr{}
\email{jujun@math.jhu.edu}


\subjclass[2010]{14T99(primary), 18G60(secondary)}

\keywords{Characteristic one, Semiring scheme, \v{C}ech cohomology, Picard group.}

\date{}

\dedicatory{}

\begin{abstract}
\normalsize{\noindent A semiring scheme generalizes a scheme in such a way that the underlying algebra is that of semirings. We generalize \v{C}ech cohomology theory and invertible sheaves to semiring schemes. In particular, when $X=\mathbb{P}^n_M$, a projective space over a totally ordered idempotent semifield $M$, we show that $\textnormal{\v{H}}^m(X,\mathcal{O}_X)$ is in agreement with the classical computation for all $m$. Finally, we classify all invertible sheaves on $X=\mathbb{P}^n_M$ by computing $\Pic(X)$ explicitly.}
\end{abstract}

\maketitle

\section{Introduction}
In this paper, our main interest is \v{C}ech cohomology theory of a semiring scheme which is a generalization of a scheme based on commutative semirings. A notion of semiring schemes has been known in relation to $\mathbb{F}_1$-geometry (cf. \cite{semibook}, \cite{oliver1}), however there are very few results on semiring schemes. In particular, sheaves and homological methods on semiring schemes have been never considered.\\
In \cite{giansiracusa2013equations}, J.Giansiracusa and N.Giansiracusa proved that one can associate a semiring scheme $X$ to a tropical variety $Y$ in such a way that $Y$ can be identified with $X(\mathbb{R}_{max})$, the set of `$\mathbb{R}_{max}$-rational points' of $X$, where $\mathbb{R}_{max}=(\mathbb{R}\cup\{-\infty\},\max,+)$ is a tropical semifield with a maximum convention (also, see \cite{giansiracusa2014universal} for the further development in connection to the Berkovich analytification). This opens the door to approach tropical geometry by means of semiring schemes and, to this end, one needs to better comprehend semiring schemes in perspective of both $\mathbb{F}_1$-geometry and tropical geometry. This paper is organized as follows:\\
In $\S 2$, we quickly review basic properties of semiring schemes and then in $\S 3$ we use a tensor product of semimodules defined by B.Pareigis and H.Rohrl in \cite{remarksonsemimodules} to confirm that a construction of Picard groups can be generalized to semiring schemes. Finally, in $\S 4$, we generalize \v{C}ech cohomology theory to semiring schemes by appealing to the framework of A.Patchkoria in \cite{cech}. The basic idea is to replace a coboundary map with a pair of coboundary maps. The following is the main result of the paper.
\begin{nothm}(cf. Propositions \ref{cechglobal}, \ref{cechforPn}, \ref{cechprojectivespace}, Theorem \ref{cech and picard equal}, Corollary \ref{invsheafclass})\\
Let $X$ be a semiring scheme. Then the following hold.
\begin{enumerate}
\item
$\Gamma(X,\mathcal{O}_X) \simeq \textnormal{\v{H}}^0(X,\mathcal{O}_X)$. 
\item
$\Pic(X)\simeq \textnormal{\v{H}}^1(X,\mathcal{O}_X^*)$. 
\item
Let $X$ be a projective space $\mathbb{P}^n_M$ over a totally ordered idempotent semifield $M$. Then, for the standard open covering $\mathcal{U}:=\{D(x_0),...,D(x_n)\}$ of $X$, we have 
\[\textnormal{\v{H}}^0(X,\mathcal{O}_X)\simeq M,\quad \textnormal{\v{H}}^p(X,\mathcal{U},\mathcal{O}_X) =0 \textrm{ }\forall p \geq 1, \quad \Pic(X)\simeq \mathbb{Z}.\]
In particular, any invertible sheaf on $X$ is isomorphic to $\mathcal{O}_X(m)$ for some $m \in \mathbb{Z}$.
\end{enumerate}
\end{nothm}
\subsection*{Acknowledgment}
This is a part of the author's Ph.D. thesis \cite{jaiungthesis}. The author is grateful to his academic advisor Caterina Consani for introducing this project. The author would like to give special thanks to Jeffrey Giansiracusa for interesting conversations and suggesting further directions. $\S 4$ is reshaped under his guidance. In particular, the tropical partition of unity argument in Lemma \ref{affinevanishing} is his idea. The author also thanks Oliver Lorscheid for sharing his preprint on \v{C}ech cohomology of blue schemes. Finally, the author thanks Paul Lescot for helpful comments regarding Remark \ref{rmk}.

\section{Review: Construction of semiring schemes}\label{constofsemischeme}
Throughout the paper, all semirings are assumed to be commutative. A (multiplicatively) cancellative semiring $M$ is a semiring such that: $\forall x,y,z \in M$, $xy=xz$ implies $y=z$ if $x\neq 0_M$. Note that this is different from $M$ having no (multiplicative) zero-divisor due to the absence of additive inverses. For an idempotent semiring $M$, one has a canonical partial order arising from an addition as follows:
\begin{equation}\label{totalorder}
x \leq y \iff x+y=y, \quad x,y \in M.
\end{equation}
By a totally ordered idempotent semiring we always mean an idempotent semiring such that a canonical partial order \eqref{totalorder} is totally ordered. For an introduction to semiring theory, we refer the reader to \cite{semibook} (also, see Appendix \ref{semirings} for the basic definitions). \\
Recall that for a semiring $M$, by a prime ideal $\mathfrak{p}$ of $M$ we mean an ideal $\mathfrak{p}$ of a semiring $M$ such that if $xy \in \mathfrak{p}$, then $x\in \mathfrak{p}$ or $y \in \mathfrak{p}$. The set $X=\Spec M$ is a topological space equipped with Zariski topology. Then, as in the classical case, we can implement the structure sheaf $\mathcal{O}_X$ of $X$ to obtain a semiring scheme. The following is well known in the theory of semiring schemes (cf. \cite{oliver1}, \cite{toen2009dessous}).

\begin{pro}\label{semiglobalsection}
Let $M$ be a semiring and $X=\Spec M$ be an affine semiring scheme. 
\begin{enumerate}
\item
For a non-zero element $f \in M$, we have $M_f \simeq \mathcal{O}_X(D(f))$. In particular, $M \simeq \mathcal{O}_X(X)$. 
\item
For $\mathfrak{p} \in X=\Spec M$, the stalk $\mathcal{O}_{X,\mathfrak{p}}$ of the sheaf $\mathcal{O}_X$ is isomorphic to the local semiring $M_\mathfrak{p}$.
\item
The opposite category of affine semiring schemes is equivalent to the category of semirings.
\end{enumerate} 
\end{pro}

\begin{rmk}\label{rmk}
In the papers, \cite{les1},\cite{les2},\cite{les3}, P.Lescot considered a topological space of prime congruences instead of prime ideals. Let $M$ be a semiring. A congruence on $M$ is an equivalence relation preserving operations of $M$. More precisely, if $x \sim y$ and $a \sim b$, then $xa \sim yb$ and $x+a \sim y+b$ $\forall x,y,a,b \in M$. A prime congruence is a congruence $\sim$ which satisfies the following condition: if $xy\sim 0$, then $x \sim 0$ or $y \sim 0$. In the theory of commutative rings, there is a one to one correspondence between congruences on a commutative ring $A$ and ideals of $A$. However, such a correspondence no longer holds for semirings. In general, one only obtains an ideal from a congruence as follows: 
\begin{equation}\label{idealfromcongruenc}
I_{\sim}:=\{a \in M \mid a \sim 0\}.
\end{equation}
The main advantage of a congruence over an ideal is that in the theory of semirings a quotient by an ideal does not behave well, however, a quotient by a congruence behaves well.\\
Similar to the construction of a prime spectrum $\Spec M$, one can define the set $X$ of prime congruences and impose Zariski topology on $X$. Each ideal $I_{\sim}$ arises from a congruence $\sim$ as in \eqref{idealfromcongruenc} is called a saturated ideal. In his papers, Lescot had not considered a structure sheaf on the topological space $X$. However, one can mimic the construction of a structure sheaf on semiring schemes by using saturated prime ideals to construct a structure sheaf for a topological space of congruence spectra. This might give a notion of a congruence semiring scheme $(X,\mathcal{O}_X)$. It appears, however, that a semiring $\mathcal{O}_X(X)$ of global sections of an `affine congruence semiring scheme $(X,\mathcal{O}_X)$' might not be isomorphic to a semiring $M$ since a naive generalization of Hilbert's Nullstellensatz (which is the main ingredient in the proof of the classical case) does not hold in the case of congruences. If every ideal of a semiring $M$ is saturated, then an affine semiring scheme induced from $M$ and an affine congruence semiring scheme induced from $M$ are isomorphic as locally semiringed spaces. For example, this is the case when $M$ is a commutative ring. Finally, we remark that a different version of tropical Nullstellensatz has been studied by D.Jo\'{o} and K.Mincheva in \cite{joo2014prime}.
\end{rmk}

\section{ Picard group of a semiring scheme}
\noindent For a given semiring scheme $X$, one defines a sheaf of $\mathcal{O}_X$-semimodules to be a sheaf $\mathcal{F}$ of sets on $X$ such that $\mathcal{F}(U)$ is an $\mathcal{O}_X(U)$-semimodule, and restriction maps $\mathcal{F}(U) \longrightarrow \mathcal{F}(V)$ and $\mathcal{O}_X(U) \longrightarrow \mathcal{O}_X(V)$ are compatible for open sets $V \subseteq U$ of $X$. A morphism of sheaves of $\mathcal{O}_X$-semimodules is also defined as in the classical case. In particular, we call a sheaf $\mathcal{L}$ of $\mathcal{O}_X$-semimodules invertible if $\mathcal{L}$ is locally isomorphic to $\mathcal{O}_X$.  

\begin{myeg}
Clearly, a structure sheaf $\mathcal{O}_X$ is a sheaf of $\mathcal{O}_X$-semimodules.
Furthermore, let $\mathcal{F},\mathcal{G}$ be sheaves of $\mathcal{O}_X$-semimodules. Then, as in the classical case, the sheaf $\Hom(\mathcal{F},\mathcal{G})$ becomes a sheaf of $\mathcal{O}_X$-semimodules.
\end{myeg}

\begin{myeg}\label{twistingsheaf}
Let $X=\mathbb{P}^n_M$ be an $n$-dimensional projective space over a semifield $M$; one may consider $X$ as $\Proj M[x_0,...,x_n]$ or equivalently, as the semiring scheme with $(n+1)$ open affine charts $U_i:=\Spec M[\frac{x_0}{x_i},...\frac{x_i}{x_i},...,\frac{x_n}{x_i}]$ for $i=0,...,n$ as in the classical case. Let $S=M[x_0,...,x_n]$ and $S(m)$ be the degree $m$-part of $S$ (which is a graded $S$-semimodule). Then, one can easily generalize the classical construction of $\mathcal{O}_X(m):=\widetilde{S(m)}$ and check from the similar argument (for example, see \cite[Proposition 5.12, \S 2.5]{Har} ) that $\mathcal{O}_X(m)$ is an invertible sheaf on $X$.
\end{myeg}

\noindent Next, we construct the tensor product $\mathcal{F} \otimes_{\mathcal{O}_X} \mathcal{G}$ of sheaves of $\mathcal{O}_X$-semimodules. Note that when we define a tensor product of semimodules, we need to be careful. There are several ways one can generalize the classical construction of a tensor product to semimodules, and some generalizations might not work well. For example, the generalization as in the Golan's book \cite{semibook} is not a proper generalization. In fact, if we follow the generalization of a tensor product in \cite{semibook}, for a semiring $A$ and an $A$-semimodule $M$, we have 
\begin{equation}\label{tensorgolan}
A \otimes_A M \simeq (M/\sim),
\end{equation}
where $\sim$ is a congruence relation on $M$ such that $a \sim b$ if and only if $\exists$ $c \in M$ such that $a+c=b+c$. When $A$ is an idempotent semiring (in which our main interest lies), the tensor product of \cite{semibook} badly behaves. For example, we have $\mathbb{Z}_{max} \otimes_{\mathbb{Z}_{max}} \mathbb{R}_{max} \simeq \{0\}$. Furthermore, we have 
\[\{0\}=\Hom(\mathbb{Z}_{max} \otimes_{\mathbb{Z}_{max}}\mathbb{Z}_{max},\mathbb{Z}_{max}) \neq \Hom(\mathbb{Z}_{max},\Hom(\mathbb{Z}_{max},\mathbb{Z}_{max}))=\mathbb{Z}_{max}. \]
This implies that we can not have the Hom-Tensor duality at the level of sheaves of $\mathcal{O}_X$-semimodules with the Golan's notion. Therefore, one can not generalize directly the construction of Picard groups. To this end, we use the definition of a tensor product which is proposed in \cite{remarksonsemimodules}. Then we recover usual isomorphisms which one can expect from a tensor product. In particular, we have $R\otimes_R M \simeq M\otimes_RR \simeq M$ and $\Hom(M\otimes_RN,P)\simeq \Hom(M,\Hom(N,P))$ for a semiring $R$ and $R$-semimodules, $M,N,P$. By appealing to such results, we define the Picard group $\Pic(X)$ of a semiring scheme $X$. 
\begin{mydef}
Let $X$ be a semiring scheme and $\mathcal{F},\mathcal{G}$ be sheaves of $\mathcal{O}_X$-semimodules. We define $\mathcal{F}\otimes_{\mathcal{O}_X}\mathcal{G}$ to be the sheafification of the presheaf $\mathcal{H}$, where $\mathcal{H}(U)=\mathcal{F}(U)\otimes_{\mathcal{O}_X(U)}\mathcal{G}(U)$ for each open set $U$ of $X$ and the tensor product is as in \cite{remarksonsemimodules}.
\end{mydef}
\begin{rmk}
One can easily observe that $\mathcal{F}\otimes_{\mathcal{O}_X}\mathcal{G}$ is indeed a sheaf of $\mathcal{O}_X$-semimodules.
\end{rmk}

The following are statements which can be directly generalized from the classical statements (mainly due to the fact that the existence of additive inverses is not used in the classical proofs).
\begin{lem}\label{stalklem}
Let $X$ be a semiring scheme. Let $\mathcal{F},\mathcal{G}$ be sheaves of $\mathcal{O}_X$-semimodules and $\mathcal{L}$ be an invertible sheaf of $\mathcal{O}_X$-semimodules on $X$.
\begin{enumerate}
\item
\[\label{tensorstalk}
(\mathcal{F} \otimes_{\mathcal{O}_X} \mathcal{G})_\mathfrak{p} \simeq \mathcal{F}_\mathfrak{p} \otimes_{\mathcal{O}_{X,\mathfrak{p}}} \mathcal{G}_\mathfrak{p}, \quad \forall \mathfrak{p} \in X
\]
\item
\[\label{homtensor}
\Hom_{\mathcal{O}_X}(\mathcal{L},\mathcal{O}_X) \otimes_{\mathcal{O}_X} \mathcal{L} \simeq \Hom_{\mathcal{O}_X}(\mathcal{L},\mathcal{L}).
\]
\item
\[
\Hom_{\mathcal{O}_X}(\mathcal{O}_X,\mathcal{O}_X ) \simeq \mathcal{O}_X.
\]
\item
\[
\Hom_{\mathcal{O}_X}(\mathcal{L},\mathcal{L}) \simeq \mathcal{O}_X.
\]
\item
The sheaf $\Hom_{\mathcal{O}_X}(\mathcal{L},\mathcal{O}_X)$ is also an invertible sheaf of $\mathcal{O}_X$-semimodules. Furthermore, we have the following isomorphism:
\begin{equation}\Hom_{\mathcal{O}_X}(\mathcal{L},\mathcal{O}_X) \otimes_{\mathcal{O}_X} \mathcal{L} \simeq \mathcal{O}_X.
\end{equation}
\end{enumerate}
\end{lem}

\begin{rmk}
It follows from Lemma \ref{stalklem} that the set $\Pic(X)$ of isomorphism classes of invertible sheaves (of $\mathcal{O}_X$-semimodules) on a semiring scheme $X$ is indeed a group with a group operation $\otimes_{\mathcal{O}_X}$ as in the classical case. In other words, in a monoid of sheaves (with a binary operation given by a tensor product) of $\mathcal{O}_X$-semimodules, the group of invertible elements are indeed sheaves which are locally isomorphic to $\mathcal{O}_X$. This justifies our term of an invertible sheaf on a semiring scheme. In the next section, we will construct \v{C}ech cohomology theory for a semiring scheme $X$, and derive the following classical result: 
\[\Pic(X) \simeq \textnormal{\v{H}}^1(X,\mathcal{O}_X^*).\]  
\end{rmk}


\section{\v{C}ech cohomology}
In \cite{cech}, A.Patchkoria generalized the notion of a chain complex of modules to semimodules by realizing that an alternating sum can be written as the sum of two sums in such a way that one stands for a positive sum and the other a negative sum. In this section, we use this idea to define \v{C}ech cohomology with values in sheaves of semimodules. Then we compute the simple case of a projective space $\mathbb{P}^n_M$ over a totally ordered idempotent semifield $M$. 

\begin{rmk}
One might be also interested in developing the sheaf cohomology for semiring schemes via derived functors. In \cite{jaiungthesis}, we proved that an idempotent semimodule (as well as a sheaf of idempotent semimodules on a semiring scheme) has a (properly defined) injective resolution. However, different from the classical case, the global section functor is not left exact. Moreover, it is unclear whether any two injective resolutions are homotopic (in a suitable sense) or not. There is some evidence that the derived functors approach to the sheaf cohomology might not be a good direction to pursue. More precisely, in \cite{oliver3}, Lorscheid computed the sheaf cohomology of the projective line $\mathbb{P}^1_{\mathbb{F}_1}$ over $\mathbb{F}_1$ via an injective resolution and found that the computation is not in accordance with the classical result. For example, $\mathrm{H}^1(\mathbb{P}^1_{\mathbb{F}_1},\mathcal{O}_{\mathbb{P}^1_{\mathbb{F}_1}})$ is an infinite-dimensional $\mathbb{F}_1$-vector space whereas classically, we have $\mathrm{H}^1(\mathbb{P}^1,\mathcal{O}_{\mathbb{P}^1})=0$. Although this is the case of a monoid scheme, this suggests that one might have to look for other possible approaches.\\ 
\end{rmk}

\begin{mydef}(cf. \cite[Definition 1.10]{cech})\label{cechcoho}
\begin{enumerate}
\item
Let $R$ be a semiring. A cochain complex (of $R$-semimodules) $X=\{X^n,\partial_n^+,\partial_n^-\}_{n\in\mathbb{Z}}$ consists of $R$-semimodules $X^n$ and $R$-homomorphisms $\partial_n^+,\partial_n^-$ as follows:
\[X:\xymatrixcolsep{3pc}\xymatrix{
\cdots \ar[r]_-{\partial_{n-2}^-} \ar@<1ex>[r]^-{\partial_{n-2}^+}
& X^{n-1} \ar[r]_-{\partial_{n-1}^-} \ar@<1ex>[r]^-{\partial_{n-1}^+} 
& X^n \ar[r]_-{\partial_{n}^-} \ar@<1ex>[r]^-{\partial_{n}^+}
& X^{n+1}\ar[r]_-{\partial_{n+1}^-} \ar@<1ex>[r]^-{\partial_{n+1}^+}
&\cdots }, \quad n \in \mathbb{Z},
\]
which satisfies the following condition:
\begin{equation}\label{alexcochain}
\partial^+_{n+1}\circ \partial^+_{n}+\partial^-_{n+1}\circ \partial^-_{n}=\partial^-_{n+1}\circ \partial^+_{n}+\partial^+_{n+1}\circ \partial^-_{n}, \quad n \in \mathbb{Z}.
\end{equation}
\item
For a cochain complex $X$, one defines the following $R$-semimodule: 
\[Z^n(X):=\{x \in X^n\mid\partial^+_{n}(x)=\partial^-_{n}(x)\}\]
as $n$-cocycles, and the $n$-th cohomology as an $R$-semimodule
\[H^n(X):=Z^n(X)/\rho^n,\]
where $\rho^n$ is a congruence relation on $Z^n(X)$ such that $x\rho^n y$ if and only if 
\begin{equation}\label{cohomologycongruence}
x+\partial^+_{n-1}(u)+\partial^-_{n-1}(v)=y+\partial^+_{n-1}(v)+\partial^-_{n-1}(u) \textrm{ for some } u,v \in X^{n-1}.
\end{equation}
\end{enumerate}
\end{mydef}
Suppose that $X=\{X^n,d_n^+,d_n^-\}$ and $Y=\{Y^n,\partial_n^+,\partial_n^-\}$ are cochain complexes of semimodules. Then, by a $\pm$-morphism from $X$ to $Y$ one means a collection $f=\{f^n\}$ of homomorphisms of semimodules which satisfies the following condition:
\begin{equation}\label{pmhomcondiotion}
f^{n+1}\circ d_n^+=\partial_n^+\circ f^n,\quad f^{n+1}\circ d_n^-=\partial_n^-\circ f^n.
\end{equation}
In \cite{cech}, it is proven that a $\pm$-morphism $f=\{f^n\}$ from $X=\{X^n,d_n^+,d_n^-\}$ to $Y=\{Y^n,\partial_n^+,\partial_n^-\}$ induces a canonical homomorphism $H^n(f)$ of cohomology semimodules as follows:
\begin{equation}\label{cohomologyinduced}
H^n(f):H^n(X) \longrightarrow H^n(Y), \quad [x] \mapsto [f^n(x)], \quad n \in \mathbb{Z},
\end{equation}
where $[x]$ is the equivalence class of $x \in Z^n(X)$ in $H^n(X)$.

\begin{rmk}\label{abgroupcech}
As pointed out in \cite{cech}, a sequence $G=\{G^n,d^+_n,d^-_n\}$ of modules is a cochain complex in the sense of Definition \ref{cechcoho} if and only if $G'=\{G^n,\partial^n:=d^+_n-d^-_n\}$ is a cochain complex of modules in the classical sense. Clearly, in this case, the cohomology semimodules of $G$ as in Definition \ref{cechcoho} is the cohomology modules of $G'$ in the classical sense.
\end{rmk}

\begin{rmk}
Since differential maps of many (co)homology theories are defined by alternating sums (simplicial methods), it seems that many of those theories can be directly generalized by using the above framework. For example, if $k$ is a semifield, then Hochschild homology can be computed via the above framework and the result is same as classical case, i.e. $HH_0(k)=k$ and $HH_n(k)=0$ for all $n>0$. 
\end{rmk}

By means of Definition \ref{cechcoho}, we introduce \v{C}ech cohomology with values in sheaves of semimodules which generalizes the classical construction. Let $R$ be a semiring, $X$ be a topological space, and $\mathcal{F}$ be a sheaf of $R$-semimodules on $X$. Suppose that $\mathcal{U}=\{U_i\}_{i\in I}$ is an open covering of $X$, where $I$ is a totally ordered set. Let $U_{i_0i_1\cdots i_p}:=U_{i_0}\cap ...\cap U_{i_p}$. We define the following set: 
\begin{equation}\label{cechcycle}
C^n=C^n(\mathcal{U},\mathcal{F}):=\prod_{i_0<...<i_n}\mathcal{F}(U_{i_0i_1\cdots i_n}), \quad n \in \mathbb{N}.
\end{equation}
Let $x_{i_0\cdots i_n}$ be the coordinate of $x \in C^n$ in $\mathcal{F}(U_{i_0i_1\cdots i_n})$. The differentials are given as follows: 
\begin{equation}\label{+map}
(d_n^+(x))_{i_0i_1\cdots i_{n+1}}=\sum_{k=0,k=\textrm{even}}^{n+1} x_{i_0\cdots \hat{i_k}\cdots i_{n+1}}|_{U_{i_0i_1\cdots i_{n+1}}},
\end{equation}
\begin{equation}\label{-map}
(d_n^-(x))_{i_0i_1\cdots i_{n+1}}=\sum_{k=0,k=\textrm{odd}}^{n+1} x_{i_0\cdots \hat{i_k}\cdots i_{n+1}}|_{U_{i_0i_1\cdots i_{n+1}}},
\end{equation}
where the notation $\hat{i_k}$ means that we omit that index.
One can directly use the classical computation to show that $C=\{C^n,d^+_n,d^-_n\}$ is a cochain complex in the sense of Definition \ref{cechcoho}. We denote the $n$-th cohomology semimodule (with respect to an open covering $\mathcal{U}$) of $C$ by \v{H}$^n(X,\mathcal{U},\mathcal{F})$ and denote by \v{H}$^n(\mathcal{U},\mathcal{F})$ when there is no possible confusion of $X$.

\begin{pro}\label{cechglobal}
Let $R$ be a semiring, $X$ be a topological space, and $\mathcal{F}$ be a sheaf of $R$-semimodules on $X$. Let $\mathcal{U}$ be an open covering of $X$. Then we have
\[\textnormal{\v{H}}^0(\mathcal{U},\mathcal{F})=\mathcal{F}(X).\]
\end{pro}
\begin{proof}
By the definition, we have \v{H}$^0(\mathcal{U},\mathcal{F}):=Z^0(\mathcal{U},\mathcal{F})/\rho^0$. Moreover, $x \rho^0 y \Longleftrightarrow x+d_{-1}^+(u)+d_{-1}^-(v)=y+d^+_{-1}(v)+d_{-1}^-(u)$ for some $u,v \in C^{-1}$. Since $C^{-1}:=0$, we have $x \rho^0 y  \iff x=y$. It follows that \v{H}$^0(\mathcal{U},\mathcal{F})=Z^0(\mathcal{U},\mathcal{F})$. Consider the following:
\[\xymatrixcolsep{3pc}\xymatrix{
&C^0=\prod_{i\in I}\mathcal{F}(U_i) \ar[r]_-{d_{0}^-} \ar@<1ex>[r]^-{d_{0}^+} 
& C^1=\prod_{i<j \in I} \mathcal{F}(U_{ij})
},
\]
where $d_0^+$ is the product of maps $\mathcal{F}(U_j) \longrightarrow \mathcal{F}(U_{ij})$ induced by the inclusion $U_{ij} \longrightarrow U_j$ and $d_0^-$ is the product of maps $\mathcal{F}(U_i) \longrightarrow \mathcal{F}(U_{ij})$ induced by the inclusion $U_{ij} \longrightarrow U_i$. Clearly, we have $Z^0(\mathcal{U},\mathcal{F})\subseteq C^0$. It follows from the inclusion $U_i \hookrightarrow X$ that we have a homomorphism $r_i:\mathcal{F}(X) \longrightarrow \mathcal{F}(U_i)$, hence the following homomorphism: 
\[r=(r_i):\mathcal{F}(X)\longrightarrow C^0.\] 
Since $\mathcal{F}$ is a sheaf, we have $\Img(r) \subseteq Z^0(\mathcal{U},\mathcal{F})$. Conversely, suppose that \[y=(y_i) \in Z^0(\mathcal{U},\mathcal{F})=\{y \in C^0=\prod_{i\in I}\mathcal{F}(U_i)\mid d_0^+(y)=d_0^-(y)\}.\] 
Then we have $y_i|_{U_{ij}}=y_j|_{U_{ij}}$. It follows that there exists a unique global section $y_X \in \mathcal{F}(X)$ such that $(y_X)|_{U_i}=y_i$. Consider the following map:
\[s:Z^0(\mathcal{U},\mathcal{F})\longrightarrow \mathcal{F}(X), \quad y \mapsto y_X.\]
Then $s$ is clearly an $R$-homomorphism. Furthermore, $r\circ s$ and $s \circ r$ are identity maps. This shows that \v{H}$^0(\mathcal{U},\mathcal{F})=\mathcal{F}(X)$ for an open covering $\mathcal{U}$ of $X$.
\end{proof}

\begin{pro}\label{cechvanishing}
Let $R$ be a semiring, $X$ be a topological space, and $\mathcal{F}$ be a sheaf of $R$-semimodules on $X$. Let $\mathcal{U}$ be an open covering of $X$ which consists of $n$ proper open subsets of $X$. Then $\textnormal{\v{H}}^m(\mathcal{U},\mathcal{F})=0$ $\forall m \geq n$.
\end{pro}
\begin{proof}
The proof is identical to that of the classical case since $C^m=0$ for $m \geq n$.
\end{proof}

Recall that a covering $\mathcal{V}=\{V_j\}_{j\in J}$ of a topological space $X$ is a refinement of a covering $\mathcal{U}=\{U_i\}_{i\in I}$ if there exists a map $\sigma:J \longrightarrow I$ such that $V_j \subseteq U_{\sigma(j)}$ for each $j \in J$. Suppose that $X^n:=C^n(\mathcal{U},\mathcal{F})$ and $Y^n:=C^n(\mathcal{V},\mathcal{F})$. Then the map $\sigma$ induces the following $\pm$-morphism: 
\begin{equation}\label{pm}
\sigma^n:X^n \longrightarrow Y^n, \quad \sigma^n(x)_{j_0\cdots j_n}=x_{\sigma(j_0) \cdots\sigma(j_n)}|_{V_{j_0 \cdots j_n}}.
\end{equation}
In fact, let $X=\{X^n,d_n^+,d_n^-\}$ and $Y=\{Y^n,\partial_n^+,\partial_n^-\}$. We have
\[(\sigma^{n+1}\circ d_n^+(x))_{j_0 \cdots j_{n+1}}=
(d_n^+(x))_{\sigma(j_0)\cdots\sigma(j_{n+1})}|_{V_{j_0 \cdots j_{n+1}}}\]
\[=(\sum_{k=0,k=even}^{n+1} x_{\sigma(j_0) \cdots\hat{\sigma(j_k)}\cdots\sigma(j_{n+1})}|_{U_{\sigma(j_0)\cdots\sigma(j_{n+1})}})|_{V_{j_0\cdots j_{n+1}}}
\]
\[=(\sum_{k=0,k=even}^{n+1} x_{\sigma(j_0)\cdots\hat{\sigma(j_k)}\cdots\sigma(j_{n+1})})|_{V_{j_0\cdots j_{n+1}}}=\sum_{k=0,k=even}^{n+1}\sigma^n(x)_{j_0\cdots\hat{j_k}\cdots j_{n+1}}|_{V_{j_0\cdots j_{n+1}}}\]
\[=(\partial_n^+\circ \sigma^n(x))_{j_0\cdots j_{n+1}}.
\]
Hence, we obtain $\sigma^{n+1}\circ d_n^+=\partial_n^+\circ \sigma^n$. Similarly one can prove that $\sigma^{n+1}\circ d_n^-=\partial_n^-\circ \sigma^n$. The $\pm$-morphism $\sigma=\{\sigma^n\}$ induces a homomorphism, \v{H}$^n(\mathcal{U},\mathcal{F}) \longrightarrow$ \v{H}$^n(\mathcal{V},\mathcal{F})$.\\ 
The collection of open coverings of a topological space $X$ becomes a directed system (with a refinement as a partial order). Since (co)limits exist in the category of semimodules, the following is well defined.
\begin{mydef}
Let $R$ be a semiring. Let $X$ be a topological space and $\mathcal{F}$ be a sheaf of $R$-semimodules on $X$. We define the $n$-th \v{C}ech cohomology of $X$ with values in $\mathcal{F}$ as follows: 
\[\textnormal{\v{H}}^n(X,\mathcal{F}):=\varinjlim_{\mathcal{U}} \textnormal{\v{H}}^n(\mathcal{U},\mathcal{F}).
\]
\end{mydef}
Note that from Proposition \ref{cechglobal}, we have \v{H}$^0(X,\mathcal{F})=\mathcal{F}(X)$.

\begin{myeg}
Consider the projective line $X=\mathbb{P}^1_M$ over an idempotent semifield $M$. More precisely, we consider $X$ as the semiring scheme with two open affine charts $U_0:=\Spec M[T]$ and $U_1:=\Spec M[\frac{1}{T}]$ glued along $T \mapsto \frac{1}{T}$. As in the classical case, one observes that $\mathcal{O}_X(X)=M$. From Proposition \ref{cechglobal}, we have $\textnormal{\v{H}}^0(X,\mathcal{O}_X)=M$. Furthermore, since $X$ has the open covering $\mathcal{U}=\{U_0,U_1\}$ which consists of two proper open subsets of $X$, we have $\textnormal{\v{H}}^n(\mathcal{U},\mathcal{O}_X)=0$ for $n \geq 2$ from Proposition \ref{cechvanishing}. Finally, with respect to the covering $\mathcal{U}=\{U_0,U_1\}$, we have
\[C:\xymatrixcolsep{3pc}\xymatrix{
 M[T]\oplus M[\frac{1}{T}] \ar[r]_-{d_{0}^-} \ar@<1ex>[r]^-{d_{0}^+} 
& M[T,\frac{1}{T}] \ar[r]_-{d_{1}^-} \ar@<1ex>[r]^-{d_{1}^+}
& 0 }, 
\]
where $d_0^+(a,b)=b$ and $d_0^-(a,b)=a$. It follows that $Z^1(\mathcal{U},\mathcal{O}_X)=M[T,\frac{1}{T}]$. Let $x,y \in Z^1(\mathcal{U},\mathcal{O}_X)$. Then, we can write $x=x_0+x_1,y=y_0+y_1$, where $x_0,y_0 \in M[T]$ and $x_1,y_1 \in M[\frac{1}{T}]$. Let $u=(x_0,y_1),v=(y_0,x_1)$. Then, we have 
\[x+d_0^+(u)+d_0^-(v)=y+d_0^+(v)+d_0^-(u).\]
It follows that $x \rho^1 y$ and hence $\textnormal{\v{H}}^1(\mathcal{U},\mathcal{O}_X)=0$. However, since this computation depends on the specific covering $\mathcal{U}$, we do not know yet whether $\textnormal{\v{H}}^n(X,\mathcal{O}_X)=\textnormal{\v{H}}^n(\mathcal{U},\mathcal{O}_X)$ or not for $n \geq 1$. 
\end{myeg}

In fact, we have the following.
\begin{pro}\label{cechforPn}
Let $X=\mathbb{P}^n_M$ be an $n$-dimensional projective space over an idempotent semifield $M$. Then: 
\[\textnormal{\v{H}}^0(X,\mathcal{O}_X)\simeq M, \quad   \textnormal{\v{H}}^m(X,\mathcal{U},\mathcal{O}_X) =0 \textrm{ for all }m\geq 1,\]
where $\mathcal{U}$ is the standard open covering which consists of principal open affine sets $D(x_i)$. 
\end{pro}
\begin{proof}
From Propositions \ref{cechglobal} and \ref{cechvanishing}, we only have to show that $\textnormal{\v{H}}^r(X,\mathcal{U},\mathcal{O}_X) =0$ for all $1 \leq r \leq m-1$. Let's fix $r$. It is enough to show that for $t \in Z^r(\mathcal{U},\mathcal{O}_X)$, there exist $u,v \in C^{r-1}(\mathcal{U},\mathcal{O}_X)$ such that 
\begin{equation}\label{cocylecon}
t+d_{r-1}^+(u)+d_{r-1}^-(v)=d_{r-1}^+(v)+d_{r-1}^-(u).
\end{equation}
Let $t=(t_l)$, where $l$ is an ordered $(r+1)$-tuple in $\{0,1,\dots,m\}^{(r+1)}$. With the covering $\mathcal{U}$, we can write $t_l=\sum_i t_{l_i}$, where $t_{l_i}$ is an element of $C^{r-1}(X)$. We define $u=(u_J)$ as follows: $u_J=\sum_{l}\sum_{l_i=J}t_{l_i}$ and let $v=u$. Then, since $M$ is an idempotent, we have for each ordered $(r+1)$-tuple $l$ that
\[t_l+(d_{r-1}^+(u))_l+(d_{r-1}^-(v))_l=(d_{r-1}^+(u))_l+(d_{r-1}^-(v))_l=(d_{r-1}^+(v))_l+(d_{r-1}^-(u))_l\]
This implies \eqref{cocylecon} as we desired.
\end{proof}

\begin{rmk}
One can easily observe that from the similar argument as in Proposition \ref{cechforPn}, when $\mathcal{F}$ is a flasque sheaf of semimodules on a semiring scheme $X$, we have the classical vanishing result for $\mathcal{F}$.
\end{rmk}

\noindent Next, we prove that the Picard group $\Pic(X)$ of a semiring scheme $X$ is isomorphic to the first \v{C}ech cohomology group of the sheaf $\mathcal{O}_X^*$. The proof is not much different from the classical case, but we include the proof for completeness. Note that $\mathcal{O}_X^*$ is the sheaf such that $\mathcal{O}_X^*(U)=\{a \in \mathcal{O}_X(U) \mid ab=1 \textrm{ for some } b \in \mathcal{O}_X(U)\}$ for an open subset $U$ of $X$. Even though $\mathcal{O}_X$ is a sheaf of semirings, $\mathcal{O}^*_X$ is a sheaf of (multiplicative) abelian groups. Hence, \v{H}$^1(\mathcal{U},\mathcal{O}^*_X)$ is an abelian group. We use the multiplicative notation for $\mathcal{O}^*_X$.\\ 
In what follows, let $X$ be a semiring scheme, $\mathcal{L}$ be an invertible sheaf of $\mathcal{O}_X$-semimodules on $X$, and $\mathcal{U}=\{U_i\}_{i\in I}$ be a covering of $X$ such that $\varphi_i:\mathcal{O}_X|_{U_i}\simeq \mathcal{L}|_{U_i}$ $\forall i \in I$. Let $e_i \in \mathcal{L}(U_i)$ be the image of $1 \in \mathcal{O}_X(U_i)$ under $\varphi_i(U_i)$. Through the following lemmas, we define a corresponding cocyle in \v{H}$^1(X,\mathcal{O}^*_X)$ for an invertible sheaf $\mathcal{L}$ on $X$. 
\begin{lem}\label{picke}
For $i<j \in I$ and $U_{ij}=U_i \cap U_j$, there exists $f_{ij} \in \mathcal{O}^*_X(U_{ij})$ such that 
\[e_i|_{U_{ij}}=(e_j|_{U_{ij}})f_{ij}.\]
\end{lem}
\begin{proof}
This is clear since $e_i|_{U_{ij}}$ and $e_j|_{U_{ij}}$ are invertible elements in $\mathcal{O}^*_X(U_{ij})$.
\end{proof}

\noindent We fix $f_{ij}$ in Lemma \ref{picke}. We have the following:
\begin{lem}\label{canonicalf}
Let $f:=(f_{ij}) \in C^1(\mathcal{U},\mathcal{O}^*_X)$. Then we have $d_1^+(f)=d_1^-(f)$ and hence $f \in Z^1(\mathcal{U},\mathcal{O}^*_X)$. In particular, $f$ has a canonical image in $\textnormal{\v{H}}^1(\mathcal{U},\mathcal{O}^*_X)$.
\end{lem}
\begin{proof}
For $i<j<k$, we have $e_i|_{U_{ij}}=(e_j|_{U_{ij}})f_{ij}$, $e_j|_{U_{jk}}=(e_k|_{U_{jk}})f_{jk}$. Thus we have
\[e_i|_{U_{ijk}}=(e_j|_{U_{ijk}})(f_{ij})|_{U_{ijk}}=
(e_k|_{U_{ijk}})(f_{jk})|_{U_{ijk}}(f_{ij})|_{U_{ijk}}=
e_k|_{U_{ijk}}(f_{ik})|_{U_{ijk}}.
\]
This implies that $(f_{jk})|_{U_{ijk}}(f_{ij})|_{U_{ijk}}=(f_{ik})|_{U_{ijk}}$. It follows that $(d^+_1(f))_{ijk}=(d^-_1(f))|_{ijk}$ and hence $f=(f_{ij}) \in Z^1(\mathcal{U},\mathcal{O}^*_X)$. Therefore, $f$ has a canonical image in \v{H}$^1(\mathcal{U},\mathcal{O}^*_X)$.
\end{proof}

\begin{lem}\label{doesnotchoicee_i}
The canonical image of $f \in C^1(\mathcal{U},\mathcal{O}_X^*)$ in $\textnormal{\v{H}}^1(\mathcal{U},\mathcal{O}^*_X)$ as in Lemma \ref{canonicalf} does not depend on the choice of $e_i$.
\end{lem}
\begin{proof}
Let $\{e_i'\}_{i\in I}$ be another choice with $\{f'_{ij}\}$. We can take $\{g_i\}_{i\in I}$, where $g_i \in \mathcal{O}^*_X(U_i)$ such that $e_i'=g_ie_i$. Then, we have 
$e_i|_{U_{ij}}=f_{ij}e_j|_{U_{ij}}$, $e_i'|_{U_{ij}}=f'_{ij}e'_j|_{U_{ij}}$. It follows that $g_i|_{U_{ij}}e_i|_{U_{ij}}=f'_{ij}e_j'|_{U_{ij}}=f'_{ij}g_j|_{U_{ij}}e_j|_{U_{ij}}$ and $g_i|_{U_{ij}}e_i|_{U_{ij}}=g_i|_{U_{ij}}f_{ij}e_j|_{U_{ij}}$. Therefore, $f_{ij}g_i|_{U_{ij}}=f'_{ij}g_j|_{U_{ij}}$. This implies that $f\cdot d^-_0(g)=f'\cdot d^+_0(g)$. In other words, $f$ and $f'$ give the same canonical image in \v{H}$^1(\mathcal{U},\mathcal{O}^*_X)$. 
\end{proof}

\noindent We denote the canonical image of $f \in C^1(\mathcal{U},\mathcal{O}_X^*)$ in \v{H}$^1(\mathcal{U},\mathcal{O}^*_X)$ by $\phi_{\mathcal{U}}(\mathcal{L})$. Let $\mathcal{U}=\{U_i\}_{i\in I}$ and $\mathcal{U}'=\{V_j\}_{j \in J}$ be two open coverings of $X$ such that $\mathcal{L}|_{U_i}\simeq \mathcal{O}_X|_{U_i}$ and $\mathcal{L}|_{V_j}\simeq \mathcal{O}_X|_{V_j}$ $\forall i\in I$, $j\in J$. We define a new covering $\mathcal{U}\cap \mathcal{U}':=\{U_i \cap V_j\}_{(i,j) \in I\times J}$ of $X$. Then, clearly $\mathcal{U}\cap\mathcal{U}'$ is a refinement of $\mathcal{U}$. It follows that $\phi_{\mathcal{U}}(\mathcal{L})$ has a canonical image in $\textnormal{\v{H}}^1(\mathcal{U}\cap \mathcal{U}',\mathcal{O}^*_X)$.

\begin{lem}
Let $\mathcal{U}=\{U_i\}_{i\in I}$ and $\mathcal{U}'=\{V_j\}_{j \in J}$ be two open coverings of $X$ such that $\mathcal{L}|_{U_i}\simeq \mathcal{O}_X|_{U_i}$ and $\mathcal{L}|_{V_j}\simeq \mathcal{O}_X|_{V_j}$ $\forall i\in I$, $j\in J$. Let $f \in C^1(\mathcal{U},\mathcal{O}_X^*)$ and $f' \in C^1(\mathcal{U'},\mathcal{O}_X^*)$ (as in Lemma \ref{canonicalf}). Then the canonical images of $f$ and $f'$ are same in $\textnormal{\v{H}}^1(\mathcal{U}\cap \mathcal{U}',\mathcal{O}^*_X)$. In particular, each invertible sheaf $\mathcal{L}$ determines a unique element $\phi(\mathcal{L})$ in $\textnormal{\v{H}}^1(X,\mathcal{O}^*_X)$.
\end{lem}
\begin{proof}
Let $\{e_i\}_{i\in I}$, $\{f_{ij}\}_{i,j \in I}$ for $\mathcal{U}$ and $\{e_j'\}_{j\in J}$, $\{f'_{kl}\}_{k,l \in J}$ for $\mathcal{U}'$ as in Lemma \ref{picke}. We claim that the images of $\phi_{\mathcal{U}}(\mathcal{L})$ and $\phi_{\mathcal{U}'}(\mathcal{L})$ in \v{H}$^1(\mathcal{U}\cap\mathcal{U}',\mathcal{O}^*_X)$ are equal. Indeed, we can find $g_{ik} \in \mathcal{O}^*_X(U_i \cap V_k)$ such that $e'_k|_{U_i\cap V_k}=(g_{ik})e_i|_{U_i\cap V_k}$. Hence, from the relation $e'_k|_{U_{ij}\cap V_{kl}}=(f_{kl}')|_{U_{ij}\cap V_{kl}} \cdot e'_l|_{U_{ij}\cap V_{kl}}$, we have that $(g_{ik})|_{U_{ij}\cap V_{kl}} \cdot e_i|_{U_{ij}\cap V_{kl}}=(f_{kl}')|_{U_{ij}\cap V_{kl}}\cdot e'_l|_{U_{ij}\cap V_{kl}}=(f'_{kl})|_{U_{ij}\cap V_{kl}} \cdot (g_{jl})|_{U_{ij}\cap V_{kl}} \cdot e_j|_{U_{ij}\cap V_{kl}}$. It follows that
\begin{equation}\label{eqn1}
(g_{ik})|_{U_{ij}\cap V_{kl}} \cdot (f_{ij})|_{U_{ij}\cap V_{kl}}=(f_{kl}')|_{U_{ij}\cap V_{kl}} \cdot (g_{jl})|_{U_{ij}\cap V_{kl}}. 
\end{equation}
Let $g=(g_{ik})$ for $i \in I, k\in J$. Then, we have $g \in C^0(\mathcal{U}\cap\mathcal{U}',\mathcal{O}^*_X)$. Give the set $I\times J$ a dictionary order. Then we have
\begin{equation}\label{eqn2}
(d^+_0(g))|_{(i,k)\times (j,l)}=g_{jl}|_{U_{ij}\cap V_{kl}} \textrm{ and } (d^-_0(g))|_{(i,k)\times (j,l)}=g_{ik}|_{U_{ij}\cap V_{kl}}. 
\end{equation}
Let $\alpha:Z^1(\mathcal{U},\mathcal{O}^*_X) \longrightarrow Z^1(\mathcal{U}\cap \mathcal{U}',\mathcal{O}^*_X)$ be the $\pm$-morphism as in \eqref{pm}. Then $\alpha$ induces the map $\hat{\alpha}:$\v{H}$^1(\mathcal{U},\mathcal{O}^*_X) \longrightarrow$\v{H}$^1(\mathcal{U}\cap \mathcal{U}',\mathcal{O}^*_X)$. Similarly, for $\mathcal{U}'$, we obtain \[\beta:Z^1(\mathcal{U}',\mathcal{O}^*_X) \longrightarrow Z^1(\mathcal{U}\cap \mathcal{U}',\mathcal{O}^*_X),\quad  \hat{\beta}:\textnormal{\v{H}}^1(\mathcal{U}',\mathcal{O}^*_X) \longrightarrow\textnormal{\v{H}}^1(\mathcal{U}\cap \mathcal{U}',\mathcal{O}^*_X).\] 
In particular, if $\phi_{\mathcal{U}}(\mathcal{L})=[f]$, then $\hat{\alpha}([f])=[\alpha(f)]$, where $[f]$ is the equivalence class of $f \in Z^1(\mathcal{U},\mathcal{O}_X^*)$ in \v{H}$^1(\mathcal{U},\mathcal{O}_X^*)$. To complete the proof, we have to show that $[\alpha(f)]=[\beta(f')]$. We know that $\alpha(f)_{(i,k)\times (j,l)}=f_{ij}|_{U_{ij}\cap V_{kl}}$ and $\beta(f')_{(i,k)\times (j,l)}=f'_{kl}|_{U_{ij}\cap V_{kl}}$. It follows from \eqref{eqn1} and \eqref{eqn2} that \[(\alpha(f)\cdot d_0^-(g))|_{U_{ij}\cap V_{kl}}=(\beta(f')\cdot d_0^+(g))|_{U_{ij}\cap V_{kl}}.\]
This proves that $[\alpha(f)]=[\beta(f')]$. Thus, $f$ and $f'$ have the same image in \v{H}$^1(X,\mathcal{O}^*_X)$. We denote this image by $\phi(\mathcal{L})$.
\end{proof}

\noindent Consider the following map: 
\begin{equation}\label{picisotoH1}
\phi:\Pic(X) \longrightarrow\textnormal{\v{H}}^1(X,\mathcal{O}^*_X), \quad [\mathcal{L}]\mapsto \phi(\mathcal{L}),
\end{equation}
where $[\mathcal{L}]$ is the isomorphism class of $\mathcal{L}$ in $\Pic(X)$.

\begin{lem}
$\phi$ is well defined.
\end{lem}
\begin{proof}
Suppose that $\mathcal{L}\simeq \mathcal{L}'$. We have to show that $\phi(\mathcal{L})=\phi(\mathcal{L}')$. Let us fix an isomorphism $\varphi:\mathcal{L} \longrightarrow \mathcal{L}'$. We can find an open covering $\mathcal{U}=\{U_i\}_{i\in I}$ of $X$ such that on $U_i$ both $\mathcal{L}$ and $\mathcal{L}'$ are isomorphic to $\mathcal{O}_X$. Let $\{e_i\}$ and $\{f_{ij}\}$ be as in Lemma \ref{picke} for $\mathcal{L}$. Then we have $\varphi_{U_{ij}}(e_i)|_{U_{ij}}=f_{ij}\cdot \varphi_{U_{ij}}(e_j)|_{U_{ij}}$. Since $\phi(\mathcal{L}')$ does not depend on the choice of $\{e_i'\}$, we let $e_i'=\varphi_{U_i}(e_i)$ as in Lemma \ref{picke} for $\mathcal{L}'$. Then the desired property follows.
\end{proof}
\begin{lem}
$\phi$ is a group homomorphism.
\end{lem}
\begin{proof}
Suppose that $\mathcal{L}$ and $\mathcal{L}'$ are invertible sheaves of $\mathcal{O}_X$-semimodules. Then so is $\mathcal{L}\otimes_{\mathcal{O}_X}\mathcal{L}'$ (this directly follows from Lemma \ref{stalklem}). Therefore, we can find an affine open covering $\mathcal{U}=\{U_i=\Spec R_i\}_{i \in I}$ of $X$ such that $(\mathcal{L}\otimes_{\mathcal{O}_X}\mathcal{L}')(U_i) \simeq \mathcal{O}_X(U_i)\simeq\mathcal{L}(U_i)\simeq\mathcal{L}'(U_i)\simeq R_i$. In particular, we have $(\mathcal{L}\otimes_{\mathcal{O}_X}\mathcal{L}')(U_i) \simeq (\mathcal{L}(U_i)\otimes_{\mathcal{O}_X}\mathcal{L}'(U_i))$. Let $\{e_i\}_{i\in I}$, $\{f_{ij}\}_{i,j \in I}$ for $\mathcal{L}$ and $\{e_j'\}_{j\in J}$, $\{f'_{kl}\}_{k,l \in J}$ for $\mathcal{L}'$ as in Lemma \ref{picke} on the open covering $\mathcal{U}$. Then we can take $\{e_i\otimes e_i'\}$ as a basis for $(\mathcal{L}\otimes_{\mathcal{O}_X}\mathcal{L}')(U_i)$ and the corresponding transition map is $F=(f_{ij}\cdot f_{ij}')$. It follows that $\phi(\mathcal{L}\otimes_{\mathcal{O}_X}\mathcal{L}')=\phi(\mathcal{L})\phi(\mathcal{L}')$.
\end{proof}

\begin{lem}
$\phi([\mathcal{L}])=1$ if and only if $[\mathcal{L}]$ is the isomorphism class of $\mathcal{O}_X$. In particular, $\phi$ is injective.
\end{lem}
\begin{proof}
Suppose that $\phi(\mathcal{L})=1$. Let $\mathcal{U}=\{U_i\}_{i \in I}$ be an open covering of $X$ such that $\mathcal{L}|_{U_i}\simeq \mathcal{O}_X|_{U_i}$ $\forall i \in I$ and let $f$ and $e_i$ be as in Lemma \ref{picke}. Since the canonical image of $f$ does not depend on the choice of an open covering $\mathcal{U}$, we may assume that $[f]=[1] \in $\v{H}$^1(\mathcal{U},\mathcal{O}^*_X)$. This implies that there exists $g \in C^0(\mathcal{U},\mathcal{O}^*_X)$ such that $d_0^+(g)=f\cdot d_0^-(g)$. Hence, $(d_0^+(g))_{ij}=(f\cdot d_0^-(g))_{ij}$ and $f_{ij}\cdot g_i|_{U_{ij}}=g_j|_{U_{ij}}$. It follows that $(g_ie_i)|_{U_{ij}}=g_i|_{U_{ij}}e_i|_{U_{ij}}=g_i|_{U_{ij}}f_{ij}e_j|_{U_{ij}}=g_j|_{U_{ij}}e_j|_{U_{ij}}=(g_je_j)|_{U_{ij}}$. Thus, $e_ig_i$ and $e_jg_j$ agree on $U_{ij}$ and hence we can glue them to obtain the global isomorphism $\varphi:\mathcal{L} \longrightarrow \mathcal{O}_X$. Conversely, if $\mathcal{L}\simeq \mathcal{O}_X$, then clearly $\phi(\mathcal{L})=1$. In fact, one can take $e_i=e|_{U_i}$, where $e$ is the identity in $\mathcal{O}_X(X)$.
\end{proof}
\begin{lem}
$\phi$ is surjective.
\end{lem}
\begin{proof}
Notice that $\alpha \in$\v{H}$^1(X,\mathcal{O}^*_X)$ comes from $[f] \in $\v{H}$^1(\mathcal{U},\mathcal{O}^*_X)$ for an open covering $\mathcal{U}=\{U_i\}_{i\in I}$ of $X$. Let $\mathcal{L}_i:=\mathcal{O}_X|_{U_i}$ for each $i \in I$. Let $f=(f_{ij}) \in Z^1(\mathcal{U},\mathcal{O}^*_X)$. Then, for $i <j$, each $f_{ij}$ defines the following isomorphism: \[\phi_{ij}:\mathcal{L}_i|_{U_{ij}} \longrightarrow \mathcal{L}_j|_{U_{ij}},\quad s \mapsto f_{ij}\cdot s.\] 
We define $\phi_{ii}:=id$. Since $f\in Z^1(\mathcal{U},\mathcal{O}^*_X)$, we have $d_1^+(f)=d_1^-(f)$. It follows that $(d_1^+(f))_{ijk}=f_{jk}\cdot f_{ij}=(d_1^-(f))_{ijk}=f_{ik}$, and $f_{ij}\cdot f_{jk}=f_{ik}$. This implies that $\phi_{ik}=\phi_{jk}\circ \phi_{ij}$ and therefore one can glue $\mathcal{L}_i$ to obtain the invertible sheaf $\mathcal{L}$. Let $e_i$ be the image of $1$ under the isomorphism $\mathcal{O}_X(U_i)\simeq \mathcal{L}(U_i)$. Then we obtain the corresponding $f=(f_{ij})$. This implies that $\phi([\mathcal{L}])=\alpha$, hence $\phi$ is surjective.
\end{proof}

\noindent Finally, we conclude the following theorem via the isomorphism $\phi$.

\begin{mythm}\label{cech and picard equal}
$\Pic(X)\simeq$\textnormal{\v{H}}$^1(X,\mathcal{O}^*_X)$ for a semiring scheme $(X,\mathcal{O}_X)$.
\end{mythm}

Next, we compute the Picard group of a projective space $\mathbb{P}^n_M$ over a totally ordered idempotent semifield $M$.

\begin{lem}\label{affinevanishing}
Let $M$ be a totally ordered idempotent semifield. Fix a monomial $g \in M[x_0,...,x_r]$. Let $A:=M[x_0,...,x_r]_g$, $X=\Spec A$, and $\mathcal{U}=\{D(f_i)\}$ be a finite covering of $X$ which consists of principal open subsets. Then \textnormal{\v{H}}$^n(\mathcal{U},\mathcal{O}_X^*)=0$ for all $n \geq 1$. 
\end{lem}
\begin{proof}
First, it follows from Hilbert's Nullstellensatz of semirings (see \cite{semibook}, \S 6) that 
\begin{equation}\label{null}
1=\sum_{i \in I}h_if_i \textrm{ for some }h_i \in A.
\end{equation}
Suppose that $h_i=\frac{a_i}{g^{m_i}}$ and $f_i=\frac{b_i}{g^{n_i}}$, where $a_i,b_i \in M[x_0,...,x_r]$ and $m_i,n_i \in \mathbb{N}$. Then \eqref{null} implies that there exist $l_1,l_2 \in \mathbb{N}$ such that
\begin{equation}\label{null2}
g^{l_1}=g^{l_2}\sum_{i \in I}a_ib_i.
\end{equation}
However, since $M$ is a totally ordered idempotent semifield and $g$ is a monomial, \eqref{null2} implies that $a_ib_i=g^t$ for some $t \in \mathbb{N}$ and $i \in I$. Hence $h_if_i \in A^*$ for some $i \in I$. In particular, $f_i$ is a unit in $A$. We fix this $i \in I$. Then $D(f_i)=X$ and for any $\{i_0,...,i_m\} \subseteq I$, we have
\begin{equation}\label{cechcondition}
\mathcal{O}^*_X(D(f_i)\cap D(f_{i_0})\cap \cdots \cap D(f_{i_m}))=\mathcal{O}^*_X(D(f_{i_0})\cap \cdots \cap D(f_{i_m})).
\end{equation}
Next, we note that since $\mathcal{O}_X^*$ is a sheaf of abelian groups, from Remark \ref{abgroupcech}, \v{C}ech cohomology groups of $X$ can be also computed by using unordered cochains as follows:
\[C^n=C^n(\mathcal{U},\mathcal{O}_X^*):=\prod_{(i_0,...,i_n)\in I^{n+1}}\mathcal{O}_X^*(U_{i_0i_1\cdots i_n}), \quad n \in \mathbb{N}.\]
We also note that in this case \v{C}ech cohomology groups can be computed by using the usual way as we mentioned in Remark \ref{abgroupcech}. We use the notation $\odot$ for the (multiplicative) group operation of $\mathcal{O}_X^*$.\\ 
Now, take any $y=(y_{i_0\cdots i_n}) \in Z^n(\mathcal{U},\mathcal{O}_X^*)$. Since we have $d^n(y)_{ii_0\cdots i_n}=1$, it follows that
\begin{equation}
y_{i_0\cdots i_n} \odot \left(\bigodot_k (y_{ii_0\cdots\hat{i_k}\cdots i_n})^{(-1)^{k+1}}\right)=1.
\end{equation}
Equivalently, we have
\begin{equation}\label{cyclecondition}
y_{i_0\cdots i_n} = \bigodot_k (y_{ii_0\cdots\hat{i_k}\cdots i_n})^{(-1)^{k}}.
\end{equation}
Let $x=(x_{i_0\cdots i_{n-1}}) \in C^{n-1}(\mathcal{U},\mathcal{O}_X^*)$, where $x_{i_0\cdots i_{n-1}}=y_{ii_0\cdots i_{n-1}}$. Note that this is possible from \eqref{cechcondition} since $\mathcal{O}^*_X(D(f_i)\cap D(f_{i_0})\cap \cdots \cap D(f_{i_{n-1}}))=\mathcal{O}^*_X(D(f_{i_0})\cap \cdots \cap D(f_{i_{n-1}}))$. This implies that
\[(d^{n-1}(x))_{i_0\cdots i_n}=\bigodot_k (x_{i_0\cdots\hat{i_k}\cdots i_n})^{(-1)^{k}}=\bigodot_k (y_{ii_0\cdots\hat{i_k}\cdots i_n})^{(-1)^{k}}=y_{i_0\cdots i_n}.
\]
Therefore we have $d^{n-1}(x)=y$ and \v{H}$^n(\mathcal{U},\mathcal{O}_X^*)=0$ for all $n\geq 1$.
\end{proof}
Next, let us recall the following well-known theorem.

\begin{mythm}[\cite{serre1955faisceaux}, Th\'{e}or\`{e}me 1 of $n^\circ$ $29$] \label{leray}
Let $X$ be a topological space, $\mathcal{U}=\{U_i\}_{i \in I}$ a covering of $X$, $\mathcal{F}$ a sheaf of abelian groups on $X$. Assume that there exists a family $(\mathcal{U}^j)_{j \in J}$ of coverings of $X$, cofinal in the family of coverings of $X$, such that \textnormal{\v{H}}$^p(\mathcal{U}^j_a,\mathcal{F}|_{U_a})=0$ for all $j \in J$, $a \in I^{n+1}$, $n \geq 0$, and $p \geq 1$. Then,
\[\textnormal{\v{H}}^n(\mathcal{U},\mathcal{F})\simeq \textnormal{\v{H}}^n(X,\mathcal{F}).\]
\end{mythm}

\begin{cor}\label{quasicoheret}
Let $X$ be an affine semiring scheme and $\mathcal{F}$ be a quasi-coherent sheaf of $\mathcal{O}_X$-semimodules which is in fact a sheaf of abelian groups. Then, \textnormal{\v{H}}$^n(X,\mathcal{F})=0$ for all $n \geq 1$. 
\end{cor}
\begin{proof}
One can easily generalize the classical proof (for example, see \cite{liu2002algebraic}) to show that for any finite covering $\mathcal{U}=\{U_i\}_{i\in I}$ of $X$ by principal open sets $U_i=D(g_i)$, \v{H}$^n(\mathcal{U},\mathcal{F})=0$ for all $n \geq 1$. This is essentially due to Hilbert's Nullstellensatz of semirings (see \cite{semibook}, \S 6). Then, by applying Theorem \ref{leray}, we obtain the desired result.
\end{proof}

\begin{cor}\label{affinevani}
Let $X$ be an affine semiring scheme as in Lemma \ref{affinevanishing}. Then \textnormal{\v{H}}$^n(X,\mathcal{O}_X^*)=0$ for $n \geq 1$.
\end{cor}
\begin{proof}
This directly follows from Theorem \ref{leray}, Remark \ref{abgroupcech}, and Lemma \ref{affinevanishing} since the family of coverings by principal open subsets is cofinal.
\end{proof}

\begin{cor}\label{coveringfree}
Let $X=\mathbb{P}^n_M$ be a projective space over a totally ordered idempotent semifield $M$. Let $\mathcal{U}=\{D(x_0),...,D(x_n)\}$ be the standard affine open covering of $X$. Then, we have
\[\textnormal{\v{H}}^m(\mathcal{U},\mathcal{O}_X^*)\simeq \textnormal{\v{H}}^m(X,\mathcal{O}_X^*), \quad m\geq 0.\]
\end{cor}
\begin{proof}
This directly follows from Corollary \ref{affinevani} and Leray's acyclicity theorem since in this case $\mathcal{U}$ is a Leray's cover.
\end{proof}

\begin{pro}\label{cechprojectivespace}
Let $X=\mathbb{P}^n_M$ be a projective space over a totally ordered idempotent semifield $M$. Then we have
\begin{equation}\label{cechprojectivespaceeqn}
\Pic(X)\simeq \textnormal{\v{H}}^1(X,\mathcal{O}_X^*)\simeq \mathbb{Z}.
\end{equation}
\end{pro}
\begin{proof}
From Corollary \ref{coveringfree}, it is enough to consider the standard covering $\mathcal{U}=\{D(x_0),...,D(x_n)\}$. Let $U_i=D(x_i)$. Then we have $\mathcal{O}_X^*(U_i)=M^*$. Let $M_i:=M^*$ for all $i=0,...,n$. As in the classical case, one can easily see that
\[\mathcal{O}_X^*(U_{ij})=\{qx_j^nx_i^{-n}\mid q \in M^*, n \in \mathbb{Z}\},\quad \mathcal{O}_X^*(U_{ijk})=\{qx_i^{n_i}x_j^{n_j}x_k^{n_k}\mid q \in M^*,n_i+n_j+n_k=0\}.\]
Let $M_{ij}:=\mathcal{O}_X^*(U_{ij})$ and $M_{ijk}:=\mathcal{O}_X^*(U_{ijk})$. We have the following \v{C}ech complex: 
\[C:\xymatrixcolsep{3pc}\xymatrix{
 C^0=\prod_i M_i \ar[r]_-{d_{0}^-} \ar@<1ex>[r]^-{d_{0}^+} 
& C^1=\prod_{i<j}M_{ij} \ar[r]_-{d_{1}^-} \ar@<1ex>[r]^-{d_{1}^+}
& C^{2}=\prod_{i<j<k}M_{ijk}\ar[r]_-{d_{2}^-} \ar@<1ex>[r]^-{d_{2}^+}
&\cdots }
\]
Then 1-cocyle is 
\begin{equation}\label{cocyle}
Z^1(\mathcal{U},\mathcal{O}_X^*)=\{x=(a_{ij})\mid d_1^+(x)=d_1^-(x)\}=\{x=(a_{ij})\mid a_{ik}|_{U_{ijk}}=a_{jk}|_{U_{ijk}} \cdot a_{ij}|_{U_{ijk}}
\}.
\end{equation}
If $a_{ij}=q_{ij}x_j^{n_{ij}}x_i^{-n_{ij}}$ then \eqref{cocyle} implies that 
\begin{equation}\label{H1condition}
q_{ik}=q_{ij}\cdot q_{jk}, \quad n_{ik}=n_{ij},\quad  n_{ik}=n_{jk}, \textrm{ and }n_{ij}=n_{jk} \quad \forall i<j<k.
\end{equation}
It follows that once $n_{01}$ is fixed then the other $n_{ij}$ will be determined by \eqref{H1condition}. Furthermore, we claim that for $x=(q_{ij}x_j^{n_{ij}}x_i^{-n_{ij}})$ and $y=(q'_{ij}x_j^{n'_{ij}}x_i^{-n'_{ij}})$ in $Z^1(\mathcal{U},\mathcal{O}_X^*)$, 
\[x \rho^1 y \iff n_{ij}=n'_{ij}\quad \forall i<j.\]
Indeed, if $x \rho^1 y $, then clearly $n_{ij}=n'_{ij}$. Conversely, we want to find elements $u=(u_i)$, $v=(v_i)$ in $C^0=\prod_iM_i$ such that 
\[x\cdot d_0^+(u)\cdot d_0^-(v)=y\cdot d_0^+(v)\cdot d_0^-(u).\]
Equivalently, for each $i<j$,
\[q_{ij}\cdot x_j^{n_{ij}}\cdot x_i^{-n_{ij}}\cdot u_i\cdot v_j=q'_{ij}\cdot x_j^{n_{ij}}\cdot x_i^{-n_{ij}}\cdot v_j\cdot u_i.\]
Since $\mathcal{O}_X^*$ is a sheaf of abelian groups we may assume that $q'_{ij}=1$. Thus, it reduces to finding $u=(u_i)$, $v=(v_i)$ such that 
\begin{equation}\label{solve}
q_{ij}\cdot u_i\cdot v_j=u_j\cdot v_i\quad \forall i<j
\end{equation}
However, the existence of such $u$ and $v$ in $C^0=\prod_iM_i$ easily follows from induction on $n$. 
\end{proof}
\begin{rmk}
Note that in the proof of Proposition \ref{cechprojectivespace} we did not use the assumption that $M$ is totally ordered and idempotent to compute $\textnormal{\v{H}}^1(\mathcal{U},\mathcal{O}_X^*)$. Such assumption is only used to prove that $\textnormal{\v{H}}^1(\mathcal{U},\mathcal{O}_X^*)=\textnormal{\v{H}}^1(X,\mathcal{O}_X^*)$ so that we can reduce the computation to the standard open cover.
\end{rmk}
\begin{rmk}
We remark that the same calculation can be done by using a reduced model of a projective space as follows. Let $M=\mathbb{R}_{max}$. Note that, different from the classical case, $\mathbb{R}_{max}[T]$ is not multiplicatively cancellative. Therefore the canonical map, $S^{-1}:\mathbb{R}_{max}[T] \longrightarrow S^{-1}\mathbb{R}_{max}[T]$ does not have to be injective. In tropical geometry, rather than working directly with $\mathbb{R}_{max}[T]$, one works with the semiring $\overline{\mathbb{R}_{max}[T]}:=\mathbb{R}_{max}[T]/\sim$, where $\sim$ is a congruence relation such that $f(T)\sim g(T) \Longleftrightarrow f(x)=g(x)$ $\forall x \in \mathbb{R}_{max}$. Let $B:=\overline{\mathbb{R}_{max}[T]}$. If $\overline{f(T)} \in B$ is multiplicatively invertible, then there exists $\overline{g(T)}$ such that $\overline{f(T)\odot g(T)}=\overline{1_B}=\overline{0}$. However, for $l \in \mathbb{R}_{max}$, the set $\overline{l}$ consists of a single element $l$. It follows that $f(T)\odot g(T)=0$. One can check that this implies that $f(T) \in \mathbb{R}$ and hence $B^*=\mathbb{R}$. Let $S=\{\overline{1},\overline{T},\overline{T}^2,...\}$ be a multiplicative subset of $B$, and $A:=S^{-1}B$. Since $B$ is multiplicatively cancellative (see \cite[\S 3.2]{jun2015valuations}), $B$ is canonically embedded into $A$. Moreover, similar to Proposition \ref{cechprojectivespace}, one can observe that $A^*=\{q\overline{T}^n \mid q \in \mathbb{R},n \in\mathbb{Z}\}$.\\ 
Suppose that the reduced model $X:=\mathbb{P}^1$ of a projective line over $\mathbb{R}_{max}$ is the semiring scheme such that two affine semiring schemes $\Spec \overline{\mathbb{R}_{max}[T]}$ and $\Spec\overline{\mathbb{R}_{max}[\frac{1}{T}]}$ are glued along $\Spec A$. The exact same argument as in the proof of Proposition \ref{cechprojectivespace} shows the following: \[\textnormal{\v{H}}^1(X,\mathcal{O}_X^*)=\mathbb{Z}.\]
\end{rmk}
Finally, we classify all invertible sheaves on a projective space over a totally ordered idempotent semifield. 
\begin{lem}\label{twisting}
Let $X=\mathbb{P}^n_M$ be a projective space over a totally ordered idempotent semifield $M$. Then for each $m \in \mathbb{Z}$, via the isomorphisms \eqref{picisotoH1} (Theorem \ref{cech and picard equal}) and \eqref{cechprojectivespaceeqn} (Proposition \ref{cechprojectivespace}), we have the isomorphism $\Pic(X) \longrightarrow \mathbb{Z}$ sending the equivalence class of $\mathcal{O}_X(m)$ to $m$. 
\end{lem}
\begin{proof}
Let $\mathcal{U}=\{D(x_0),...,D(x_n)\}$ be the standard open cover of $X$. Let $U_i:=D(x_i)$. Then, as in the classical case, one can easily observe that the map $\varphi_i:\mathcal{O}_X(m)|_{U_i} \longrightarrow \mathcal{O}_X|_{U_i}$ induced by $\psi_i:\mathcal{O}_X(U_i)\longrightarrow \mathcal{O}_X(m)(U_i)$, $a \mapsto ax_i^m$ is an isomorphism. Then, the transition map $f_{ij}$ as in Lemma \ref{picke} is given by $\frac{x_i^m}{x_j^m}$ and these $f_{ij}$ determine an element $f \in Z^1(\mathcal{U},\mathcal{O}^*_X)$. It follows from Lemmas \ref{canonicalf} and \ref{doesnotchoicee_i} that this, in turn, determines and element of \v{H}$^1(\mathcal{U},\mathcal{O}^*_X)$. On the other hand, the isomorphism \eqref{cechprojectivespaceeqn} sends the equivalence class $[f=(\frac{x_i^m}{x_j^m})]$ to $m$.
\end{proof}

\begin{cor} \label{invsheafclass}
Let $X=\mathbb{P}^n_M$ be a projective space over a totally ordered idempotent semifield $M$. Then
\begin{enumerate}
\item
$\mathcal{O}_X(l)\otimes_{\mathcal{O}_X}\mathcal{O}_X(m)\simeq \mathcal{O}_X(l+m)$  $\forall l,m \in \mathbb{Z}$.
\item
Any invertible sheaf $\mathcal{L}$ on $X$ is isomorphic to $\mathcal{O}_X(m)$ for some $m \in \mathbb{Z}$.
\end{enumerate}
\end{cor}
\begin{proof}
This is clear from Lemma \ref{twisting} since the equivalence class of $\mathcal{O}_X(1)$ in $\Pic(X)$ maps to $1 \in \mathbb{Z}$ under the isomorphism. 
\end{proof}

\appendix 
\section{Basic definitions of semirings}\label{semirings}
In this section, we provide the basic definitions of semirings which are frequently used in the paper. 
\begin{mydef}
A set $T$ equipped with a binary operation $\cdot$ is called a monoid if for $a,b,c \in T$, we have $(a\cdot b)\cdot c=a\cdot(b \cdot c)$ and there exists $1 \in T$ such that $1\cdot a=a\cdot 1=a$. When $a\cdot b=b\cdot a$ $\forall a,b \in T$, we say that $T$ is a commutative monoid. When $T$ does not have $1$, $T$ is called a semigroup.
\end{mydef}
\begin{mydef}
A semiring $(M,+,\cdot)$ is a non-empty set $M$ endowed with an addition $+$ and a multiplication $\cdot$ such that
\begin{enumerate}
\item
$(M,+)$ is a commutative monoid with the neutral element $0$.
\item
$(M,\cdot)$ is a monoid with the identity $1$.
\item
$r(s+t)=rs+rt \textrm{ and } (s+t)r=sr+tr \quad \forall r,s,t \in M.$
\item
$r\cdot 0=0\cdot r=0\quad \forall r \in M.$
\item
$0 \neq 1$.
\end{enumerate}
If $(M,\cdot)$ is a commutative monoid, then we call $M$ a commutative semiring. If $(M \backslash \{0\}, \cdot)$ is a group, then a semiring $M$ is called a semifield.
\end{mydef}

\begin{mydef}(cf. \cite{semibook})
Let $M_1$, $M_2$ be semirings. A map $f:M_1 \longrightarrow M_2$ is a homomorphism of semirings if $f$ satisfies the following conditions:
\[f(a+b)=f(a)+f(b), \quad f(ab)=f(a)f(b), \quad f(0)=0, \quad f(1)=1 \quad \forall a,b \in M_1.\] 
\end{mydef}

\begin{mydef}
Let $M$ be a commutative semiring and $T$ be a commutative monoid. We say that $T$ is a $M$-semimodule if there exists a map $\varphi:M\times T \longrightarrow T$ which satisfies the following properties: $\forall m,m_1,m_2 \in M$, $\forall t,t_1,t_2 \in T$,
\begin{enumerate}
\item
$\varphi(1,t)=t$.
\item
If $t=0$ or $m=0$, then $\varphi(m,t)=0$.
\item
$\varphi(m_1+m_2,t) = \varphi(m_1,t)+\varphi(m_2,t),\quad$ $\varphi(m,t_1+t_2) = \varphi(m,t_1)+\varphi(m,t_2)$.
\item
$\varphi(m_1m_2,t)=\varphi(m_1,\varphi(m_2,t)),\quad$ $\varphi(m,t_1t_2)=\varphi(m,t_2)t_2.$
\end{enumerate}
\end{mydef}
 
\noindent By an idempotent semiring, we mean a semiring $M$ such that $x+x=x$ $\forall x \in M$.
\begin{myeg} 
Let $\mathbb{B}:=\{0,1\}$. We define an addition as: $1+1=1$, $1+0=0+1=1$, and $0+0=0$. A multiplication is defined by $1\cdot 1=1$, $1\cdot 0=0$, and $0\cdot 0=0$. Then, $\mathbb{B}$ becomes the initial object in the category of idempotent semirings. 
\end{myeg}
\begin{myeg}
The tropical semifield $\mathbb{R}_{max}$ is $\mathbb{R}\cup \{-\infty\}$ as a set. An addition $\oplus$ is given by: $a\oplus b :=\max\{a,b\} \quad \forall a,b \in \mathbb{R}_{max}$, where $-\infty \leq a$ $\forall a \in \mathbb{R}_{max}$. A multiplication $\odot$ is defined as the usual addition of $\mathbb{R}$ as follows: $a \odot b:=a+b$ $\forall a,b \in \mathbb{R}$ and $(-\infty)\odot a=a\odot (-\infty)=(-\infty)$ $\forall a \in \mathbb{R}_{max}$. We denote by $\mathbb{Q}_{max}$, $\mathbb{Z}_{max}$ the sub-semifields of $\mathbb{R}_{max}$ with the underlying sets $\mathbb{Q}\cup \{-\infty\}$, $\mathbb{Z}\cup \{-\infty\}$ respectively.
\end{myeg}

\bibliography{Cech_Cohomology_Semiring_Scheme_Revision}\bibliographystyle{plain}
\end{document}